\documentclass[10pt]{amsart}%
\usepackage{graphicx,amscd,color,amsmath,amsfonts,amssymb,geometry}
\usepackage{amsmath}
\usepackage{amsfonts}
\usepackage{amssymb}
\usepackage{graphicx}%
\providecommand{\U}[1]{\protect\rule{.1in}{.1in}}
\newtheorem{theorem}{Theorem}[section]
\newtheorem{proposition}[theorem]{Proposition}

\newtheorem{lemma}[theorem]{Lemma}

\theoremstyle{definition}

\newtheorem{remark}[theorem]{Remark}

\thanks{D. Pellegrino is supported by CNPq and Grant 2019/0014 Paraiba State Research Foundation (FAPESQ). P. Rueda is supported by the Ministerio de Econom\'{\i}a, Industria y Competitividad and FEDER under project MTM2016-77054-C2-1-P}
\begin{document}
\title[Optimal constants of classical inequalities in complex sequence spaces]{Optimal constants of classical inequalities for complex scalars}
\author[W. V. Cavalcante]{Wasthenny V. Cavalcante}
\address{Departamento de Matem\'{a}tica \\
Universidade Federal da Para\'{\i}ba \\
58.051-900 - Jo\~{a}o Pessoa, Brazil.}
\email{wasthenny.wvc@gmail.com}
\author[D. N\'u\~nez]{Daniel N\'{u}\~{n}ez-Alarc\'{o}n}
\address{Departamento de Matem\'{a}ticas\\
Universidad Nacional de Colombia\\
111321 - Bogot\'a, Colombia}
\email{danielnunezal@gmail.com and dnuneza@unal.edu.co}
\author[D. Pellegrino]{Daniel Pellegrino}
\address{Departamento de Matem\'{a}tica \\
Universidade Federal da Para\'{\i}ba \\
58.051-900 - Jo\~{a}o Pessoa, Brazil.}
\email{pellegrino@pq.cnpq.br and dmpellegrino@gmail.com}
\author[P. Rueda]{Pilar Rueda}
\address{Departamento de An\'alisis Matem\'atico\\
Universitat de Val\`encia\\
C/ Dr. Moliner 50, 46100 Burjassot (Valencia). Spain}
\email{pilar.rueda@uv.es}
\thanks{}
\maketitle

\begin{abstract}
In this paper we obtain the optimal constants of some classical inequalities,
such as the multiple Khinchine inequality for Steinhaus variables and the
mixed Littlewood inequality for complex scalars.

\end{abstract}
\tableofcontents

\section{Introduction}

Let $\mathbb{K}$ be the real or complex scalar field. It is well-known that
the solution of optimization problems needs different techniques when dealing
with real or complex scalars. We illustrate this phenomenon with the
Bohnenblust--Hille inequalities for multilinear functionals. The
Bohnenblust--Hille inequalities assert that for every continuous $m$-linear
form $T:\mathbb{K}^{n}\times\cdots\times\mathbb{K}^{n}\rightarrow\mathbb{K}$,
the optimal constants $C_{m}^{\mathbb{K}}(n)\geq1$ in
\begin{equation}
\left(  \sum_{i_{1},\ldots,i_{m}=1}^{n}\left\vert T(e_{i_{1}},\ldots,e_{i_{m}%
})\right\vert ^{\frac{2m}{m+1}}\right)  ^{\frac{m+1}{2m}}\leq C_{m}%
^{\mathbb{K}}(n)\left\Vert T\right\Vert \label{u88}%
\end{equation}
satisfy%
\[
C_{m}^{\mathbb{K}}:=\lim_{n\rightarrow\infty}C_{m}^{\mathbb{K}}(n)<\infty.
\]
Above, $\mathbb{K}^{n}$ is considered with the $\sup$ norm and $\left\Vert
T\right\Vert =\sup\left\{  \left\vert T\left(  x^{\left(  1\right)
},...,x^{\left(  m\right)  }\right)  \right\vert :\left\Vert x^{\left(
i\right)  }\right\Vert \leq1;~1\leq i\leq m\right\}  $. The optimal values of
$C_{m}^{\mathbb{K}}(n)$ and $C_{m}^{\mathbb{K}}$ are important in different
fields of Mathematics and its applications (see, for instance \cite{ar,
montanaro}).

The best known upper bounds for $C_{m}^{\mathbb{K}}(n),C_{m}^{\mathbb{K}}$ are
presented in \cite{bayart}, as a consequence of the Khinchine inequality.
However, the optimal values of these constants are, of course, a more subtle
problem. It is convenient to observe that the Krein--Milman Theorem helps us
to re-write the optimization problem in an apparently better presentation.
From now on $\mathcal{L}\left(  ^{m}\mathbb{K}^{n}\right)  $ denotes the space
of $m$-linear forms from $\mathbb{K}^{n}\times\cdots\times\mathbb{K}^{n}$ to
$\mathbb{K}$ and $B_{\mathcal{L}\left(  ^{m}\mathbb{K}^{n}\right)  }$ denotes
its closed unit ball. Note that%
\[
C_{m}^{\mathbb{K}}(n)=\max\left\{  \left(  \sum_{i_{1},\ldots,i_{m}=1}%
^{n}\left\vert T(e_{i_{1}},\ldots,e_{i_{m}})\right\vert ^{\frac{2m}{m+1}%
}\right)  ^{\frac{m+1}{2m}}:T\in B_{\mathcal{L}\left(  ^{m}\mathbb{K}%
^{n}\right)  }\right\}
\]
and since the function%
\begin{align*}
f  &  :B_{\mathcal{L}\left(  ^{m}\mathbb{K}^{n}\right)  }\rightarrow
\mathbb{R}\\
f(T)  &  =\left(  \sum_{i_{1},\ldots,i_{m}=1}^{n}\left\vert T(e_{i_{1}}%
,\ldots,e_{i_{m}})\right\vert ^{\frac{2m}{m+1}}\right)  ^{\frac{m+1}{2m}}%
\end{align*}
is convex and $B_{\mathcal{L}\left(  ^{m}\mathbb{K}^{n}\right)  }$ is compact,
a consequence of the Krein--Milman Theorem (see \cite{cpt} for details) tells
us that
\[
C_{m}^{\mathbb{K}}(n)=\max\left\{  \left(  \sum_{i_{1},\ldots,i_{m}=1}%
^{n}\left\vert T(e_{i_{1}},\ldots,e_{i_{m}})\right\vert ^{\frac{2m}{m+1}%
}\right)  ^{\frac{m+1}{2m}}:T\in\mathcal{C}_{m,n}\right\}  ,
\]
where $\mathcal{C}_{m.n}$ is the set of extreme points of $B_{\mathcal{L}%
\left(  ^{m}\mathbb{K}^{n}\right)  }$. So, to find the exact value of
$C_{m}^{\mathbb{K}}(n)$ it is enough to know the set $\mathcal{C}_{m.n}$ and,
if this set is finite, we just need to find the maximum over a finite number
of tests. Recently, in \cite{cpt}, it was shown that the set $\mathcal{C}%
_{m.n}$ is finite in the case of real scalars, and an algorithm was
constructed, furnishing all the points of $\mathcal{C}_{m.n}$. So, with
suitable computational assistance, the exact constants $C_{m}^{\mathbb{R}}(n)$
are determined (see also \cite{vieira} for a successful implementation of the
algorithm). The complex case remains open and seems to be even more difficult,
because the geometry of the unit ball in the case of complex scalars is
apparently \textquotedblleft smoother\textquotedblright\ and the number of
extreme points is quite likely infinite. For methods for complex optimization
problems we refer, for instance, to \cite{jiang, zhang, zhang2},

For any function $f$ we shall consider $f(\infty):=\lim_{s\rightarrow\infty
}f(s)$ and for any $s\geq1$ we denote the conjugate index of $s$ by $s^{\ast
},$ i.e., $\frac{1}{s}+\frac{1}{s^{\ast}}=1$. Here $1^{\ast}$ means $\infty$.
As a matter of fact, the Bohnenblust--Hille inequalities are a particular
instance of a broader family of inequalities that we call Hardy--Littlewood inequalities:

\begin{theorem}
[Hardy--Littlewood inequalities](see \cite{abps}) Let $p_{1},...,p_{m}%
\in\left[  2,\infty\right]  $ be such that
\[
0<\frac{1}{p_{1}}+\cdots+\frac{1}{p_{m}}\leq\frac{1}{2}.
\]
If $t_{1},\ldots,t_{m}\in\left[  \frac{1}{1-\left(  \frac{1}{p_{1}}%
+\cdots+\frac{1}{p_{m}}\right)  },2\right]  $ are such that
\begin{equation}
\frac{1}{t_{1}}+\cdots+\frac{1}{t_{m}}\leq\frac{m+1}{2}-\left(  \frac{1}%
{p_{1}}+\cdots+\frac{1}{p_{m}}\right)  , \label{dez11}%
\end{equation}
then, for every $m$-linear form $A:\ell_{p_{1}}^{n}\times\cdots\times
\ell_{p_{m}}^{n}\rightarrow\mathbb{K}$, every bijection $\sigma:\left\{
1,...,m\right\}  \rightarrow\left\{  1,...,m\right\}  $ and every positive
integer $n$, the optimal constants $C_{\left(  t_{1},...,t_{m}\right)
}^{\sigma,(p_{1},...,p_{m})\mathbb{K}}(n)$ in
\begin{equation}
\left(  \sum_{i_{\sigma(1)}=1}^{n}\left(  \ldots\left(  \sum_{i_{\sigma(m)}%
=1}^{n}\left\vert A\left(  e_{i_{1}},\ldots,e_{i_{m}}\right)  \right\vert
^{t_{m}}\right)  ^{\frac{t_{m-1}}{t_{m}}}\ldots\right)  ^{\frac{t_{1}}{t_{2}}%
}\right)  ^{\frac{1}{t_{1}}}\leq C_{\left(  t_{1},...,t_{m}\right)  }%
^{\sigma,(p_{1},...,p_{m})\mathbb{K}}(n)\Vert A\Vert\label{q3}%
\end{equation}
satisfy%
\[
C_{\left(  t_{1},...,t_{m}\right)  }^{\sigma,(p_{1},...,p_{m})\mathbb{K}%
}:=\lim_{n\rightarrow\infty}C_{\left(  t_{1},...,t_{m}\right)  }%
^{\sigma,(p_{1},...,p_{m})\mathbb{K}}(n)<\infty.
\]
Moreover, the exponents are optimal.
\end{theorem}

Above, as usual, $\ell_{p}^{n}$ is $\mathbb{K}^{n}$ with the $\ell_{p}$-norm.
When $m=2,$ $p_{1}=p_{2}=\infty,$ and $\left(  t_{1},t_{2}\right)  =\left(
1,2\right)  $ or $\left(  2,1\right)  $ the optimal constants for
$\mathbb{K}=\mathbb{R}$ are $\sqrt{2}$ and for $\mathbb{K=C}$ are $\left(
2/\sqrt{\pi}\right)  $ (see \cite[page 31]{blei}). In most of the other cases
the optimal constants $C_{\left(  t_{1},...,t_{m}\right)  }^{\sigma
,(p_{1},...,p_{m})\mathbb{K}}$ and $C_{\left(  t_{1},...,t_{m}\right)
}^{\sigma,(p_{1},...,p_{m})\mathbb{K}}(n)$ are unknown, except for
$\mathbb{K}=\mathbb{R}$ and $\sigma=id$ (and all $\sigma$ when symmetry
arguments are possible), in the following particular cases (see also
\cite{alb, aron, caro} for other special cases):

\begin{enumerate}
\item[(i)] $p_{1}=\cdots=p_{m}=\infty$ $\ $and $t_{1}=\cdots=t_{m-1}=2$ and
$t_{m}=1;$

\item[(ii)] $p_{2}=\cdots=p_{m}=\infty$ $\ $and $p_{1}\in\lbrack\alpha
,\infty)$ with $\alpha\approx2.18006$ and $t_{1}=\cdots=t_{m-1}=2$ and
$t_{m}=\frac{p}{p-1};$

\item[(iii)] $p_{1}=\cdots=p_{m}=\infty$ $\ $and $t_{1}=\cdots=t_{m}%
\in\{1,2\};$

\item[(iv)] $p_{2}=\cdots=p_{m}=\infty$ $\ $and $p_{1}\in\lbrack2,\alpha)$
with $\alpha\approx2.18006$ and $t_{1}=\cdots=t_{m-1}=2$ and $t_{m}=\frac
{p}{p-1}.$
\end{enumerate}

The following table shows the optimal constants for the cases (i)-(iv),
obtained by \cite{pell, racsam, pt, ns}, respectively:\bigskip

\begin{center}%
\begin{tabular}
[c]{|l|l|l|}\hline
Case & Year & Optimal constant\\\hline
(i) & 2016, \cite{pell} & $\left(  2^{\frac{1}{2}}\right)  ^{m-1}$\\\hline
(ii) & 2017, \cite{racsam} & $\left(  2^{\frac{1}{2}-\frac{1}{p}}\right)
^{m-1}$\\\hline
(iii) & 2018, \cite{pt} & $\left(  2^{\frac{1}{2}}\right)  ^{m-1}$\\\hline
(iv) & 2019, \cite{ns} & $\left(  \frac{1}{\sqrt{2}}\left(  \frac
{\Gamma\left(  \frac{2p-1}{2p-2}\right)  }{\sqrt{\pi}}\right)  ^{\frac{1}%
{p}-1}\right)  ^{m-1}$\\\hline
\end{tabular}

\end{center}

\bigskip

Above, the exact value of $\alpha\approx2.18006$ is the following:
\[
\alpha=\frac{p_{0}}{p_{0}-1},\text{ }%
\]
where $p_{0}\in(1,2)$ is the unique real number satisfying%
\[
\Gamma\left(  \frac{p_{0}+1}{2}\right)  =\frac{\sqrt{\pi}}{2}.
\]
In the finite-dimensional case (i.e., $C_{\left(  t_{1},...,t_{m}\right)
}^{\sigma.(p_{1},...,p_{m})\mathbb{R}}(n)$ and $n<\infty$), for $p_{1}%
=\cdots=p_{m}=\infty$, the optimal constants are \textquotedblleft
formally\textquotedblright\ known, using the algorithm developed in
\cite{cpt}, but the time needed to run the algorithm is impeditive, with the
current technology (see also \cite{vieira}).

The main goal of the present paper is to obtain the optimal constants for the
cases (i)-(iv) when $\mathbb{K}=\mathbb{C}$. In Section 2 we develop an
approximation technique, fundamental for the proof of our main result. In
Section 3 we obtain the optimal constants of the multiple Khinchine inequality
for Steinhaus variables and, in the final section, we finally prove our main
result: the sharp constants for the cases (i)-(iv) when $\mathbb{K}%
=\mathbb{C}$.

\section{Approximating the $\sup$ norm of complex multilinear forms}

For $p\geq1$, we introduce the following notation: $X_{p}:=\ell_{p}\left(
\mathbb{C}\right)  =\ell_{p}$ and $X_{\infty}:=c_{0}\left(  \mathbb{C}\right)
=c_{0}$. From now on, $\left(  e_{k}\right)  _{k=1}^{\infty}$ denotes the
sequence of canonical vectors in $X_{p}$.

\smallskip Let $p_{1},\ldots,p_{m}\in\lbrack1,\infty]$. We recall that for a
continuous $m$-linear form $T:X_{p_{1}}\times\cdots\times X_{p_{m}}%
\rightarrow\mathbb{C}$, the $\sup$-norm of $T$ is given by
\begin{align*}
\left\Vert T\right\Vert  &  =\sup\left\{  \left\vert T\left(  x^{\left(
1\right)  },...,x^{\left(  m\right)  }\right)  \right\vert :\left\Vert
x^{\left(  i\right)  }\right\Vert _{X_{p_{i}}}\leq1;~1\leq i\leq m\right\} \\
&  =\sup\left\{  \left\vert \sum_{i_{1},...,i_{m}=1}^{\infty}a_{i_{1}\cdots
i_{m}}x_{i_{1}}^{\left(  1\right)  }...x_{i_{m}}^{\left(  m\right)
}\right\vert :\left\Vert x^{\left(  i\right)  }\right\Vert _{X_{p_{i}}}%
\leq1;~1\leq i\leq m\right\}  ,
\end{align*}
where $T\left(  e_{i_{1}},...,e_{i_{m}}\right)  =a_{i_{1}\cdots i_{m}}$, for
all $i_{1},...,i_{m}\in\mathbb{N}.$

In this section we present a basic lemma whose aim is to get approximations of
$\Vert T\Vert$, that will be useful in Section 4. Our approach is inspired by
\cite{jiang}.

For each integer $M\geq2$, we consider
\[
T_{M}:=\left\{  \exp\left(  \frac{2j\pi}{M}i\right)  :j=0,\ldots,M-1\right\}
,
\]
and
\[
T_{\infty}=\{\exp(ti):t\in\lbrack0,2\pi)\},
\]
where $i=\sqrt{-1}$. Observe that $T_{M}$ is the set of the $M$th roots of unity.

Let
\[
D_{M}:=conv\left(  T_{M}\right)  \text{ and }D_{\infty}:=conv\left(
T_{\infty}\right)  ,
\]
where $conv$ means the convex hull.

Observe that $D_{\infty}$ is the closed unit disk $\mathbb{D}$ and, trivially,
$D_{M}\subset D_{\infty}$. Obviously, $D_{M}$ is closed and a symmetric convex
body in $\mathbb{C}$.

\begin{lemma}
\label{apothem} Let $M\geq3$ be an integer. If $r_{M}:=\left(  \frac{1}%
{2}+\frac{1}{2}\cos\left(  \frac{2\pi}{M}\right)  \right)  ^{\frac{1}{2}} ,$
then
\[
B[0,r_{M}]\subset D_{M},
\]
where $B[0,r_{M}]$ denotes the closed ball with center in $0$ and radius
$r_{M}.$
\end{lemma}

\begin{proof}
Note that $0\in D_{M}$. In fact,
\[
0=\frac{1}{M}+\frac{1}{M}\exp\left(  \frac{2\pi}{M}i\right)  +\cdots+\frac
{1}{M}\exp\left(  \frac{2(M-1)\pi}{M}i\right)  ,
\]
and $\sum_{i=0}^{M-1}\frac{1}{M}=1$. We also know that $D_{M}$ is a regular
polygon with apothem given by
\[
\left\vert \frac{1}{2}\exp\left(  \frac{2j\pi}{M}i\right)  +\frac{1}{2}%
\exp\left(  \frac{2(j+1)\pi}{M}i\right)  \right\vert .
\]
Computing the apothem, we have%
\begin{align*}
&  \left\vert \frac{1}{2}\exp\left(  \frac{2j\pi}{M}i\right)  +\frac{1}{2}%
\exp\left(  \frac{2(j+1)\pi}{M}i\right)  \right\vert ^{2}\\
&  =\frac{1}{4}\left(  \left(  \cos\left(  \frac{2j\pi}{M}\right)
+\cos\left(  \frac{2(j+1)\pi}{M}\right)  \right)  ^{2}+\left(  \sin\left(
\frac{2j\pi}{M}\right)  +\sin\left(  \frac{2(j+1)\pi}{M}\right)  \right)
^{2}\right) \\
&  =\frac{1}{4}\left(  \left(  2+2\cos\left(  \frac{2j\pi}{M}\right)
\cos\left(  \frac{2(j+1)\pi}{M}\right)  \right)  +2\sin\left(  \frac{2j\pi}%
{M}\right)  \sin\left(  \frac{2(j+1)\pi}{M}\right)  \right) \\
&  =\frac{1}{4}\left(  2+2\cos\left(  \frac{2\pi}{M}\right)  \right) \\
&  =r_{M}^{2}.
\end{align*}
Thus, it is possible to draw a circle inside $D_{M}$ with radius $r_{M}$.
\end{proof}

Let $N$ and $M$ be positive integers, $M\geq3.$ Since $D_{M}^{N}=D_{M}%
\times\overset{N}{\cdots}\times D_{M}$ is closed and a symmetric convex body
in $\mathbb{C}$, we can to consider the Minkowski functional $p_{M,N}$
associated to the product set $D_{M}^{N}:$
\[
p_{M,N}(z)=\inf\{t>0:z\in tD_{M}^{N}\},\ \ \ z\in\mathbb{C}^{N}.
\]
It is well-known that \label{proppm} the function $p_{M,N}$ has the following properties:

\begin{enumerate}
\item[(a)] $p_{M,N}(\alpha z)=\alpha p_{M,N}(z)$ for all $\alpha>0$ and
$z\in\mathbb{C}^{N}$;

\item[(b)] $p_{M,N}(z+w)\leq p_{M,N}(z)+p_{M,N}(w)$ for any $z,w\in
\mathbb{C}^{N}$;

\item[(c)] $D_{M}^{N}=\left\{  z\in\mathbb{C}^{N}:p_{M,N}(z)\leq1\right\}  $;
\end{enumerate}

Observe that, if $z\in\mathbb{C}^{N}$, $z\not =\left(  0,...,0\right)  $, by
(a) and (c) we have $\frac{z}{p_{M,N}(z)}\in D_{M}^{N}\subseteq\mathbb{D}$,
and therefore
\[
\Vert z\Vert\leq p_{M,N}(z).
\]
Moreover, since $r_{M}\frac{z}{\Vert z\Vert}\in B[0,r_{M}]$, by Lemma
\ref{apothem} we conclude that
\[
p_{M,N}\left(  r_{M}\frac{z}{\Vert z\Vert}\right)  \leq1,
\]
or, equivalently
\[
p_{M,N}(z)\leq r_{M}^{-1}\Vert z\Vert.
\]

Summarizing, if $N,M\in\mathbb{N}$, $M\geq3$, and $z\in\mathbb{C}^{N}$ we
have
\begin{equation}
\Vert z\Vert\leq p_{M,N}(z)\leq r_{M}^{-1}\Vert z\Vert. \label{(d)}%
\end{equation}

\begin{remark}
When we allow $M\rightarrow\infty$, we have $p_{M,N}(z)\rightarrow\Vert
z\Vert$ because $r_{M}\rightarrow1$.
\end{remark}

Let $N$ and $M$ be positive integers, $M\geq3$. For any $m$-linear form
$T:X_{\infty}^{N}\times\cdots\times X_{\infty}^{N}\rightarrow\mathbb{C}$ we
consider
\[
\Vert T\Vert_{M}=\inf\left\{  C>0:|T(x^{\left(  1\right)  },...,x^{\left(
m\right)  })|\leq C\prod_{i=1}^{m}p_{M,N}(x^{\left(  i\right)  })\quad\forall
x^{\left(  i\right)  }\in\mathbb{C}^{N},i=1,\ldots m\right\}  .
\]
Clearly, for all $M\geq3$ we have%
\[
\Vert T\Vert_{M}=\sup\left\{  |T(x^{\left(  1\right)  },...,x^{\left(
m\right)  })|:x^{\left(  1\right)  },...,x^{\left(  m\right)  }\in D_{M}%
^{N}\right\}  .
\]

\begin{remark}
A straightforward consequence of the Krein-Milman Theorem gives us%
\[
\Vert T\Vert_{M}=\sup\left\{  |T(x^{\left(  1\right)  },...,x^{\left(
m\right)  })|:x^{\left(  1\right)  },...,x^{\left(  m\right)  }\in T_{M}%
^{N}\right\}  .
\]

\end{remark}

We now have all the ingredients to approximate the $\sup-$norm $\Vert.\Vert$.

\begin{proposition}
\label{estnorm} If $N,M\in\mathbb{N}$, $M\geq3$ and $T:X_{\infty}^{N}%
\times\cdots\times X_{\infty}^{N}\rightarrow\mathbb{C}$ is an $m$-linear form,
then
\[
\Vert T\Vert_{M}\leq\Vert T\Vert\leq r_{M}^{-m}\Vert T\Vert_{M},
\]
where $r_{M}$ is as in Lemma \ref{apothem}.
\end{proposition}

\begin{proof}
Since $D_{M}\subset D_{\infty}$, it is immediate that $\Vert T\Vert_{M}%
\leq\Vert T\Vert$. To prove the other inequality, note that given $x^{\left(
1\right)  },...,x^{\left(  m\right)  }\in\mathbb{C}^{N}$, we have
\[
|T(x^{\left(  1\right)  },...,x^{\left(  m\right)  })|\leq\Vert T\Vert
_{M}\prod_{i=1}^{m}p_{M,N}(x^{\left(  i\right)  })\leq\Vert T\Vert_{M}\left(
r_{M}^{-m}\prod_{i=1}^{m}\Vert x^{\left(  i\right)  }\Vert\right)  .
\]
Therefore
\[
\Vert T\Vert\leq r_{M}^{-m}\Vert T\Vert_{M}.
\]

\end{proof}

For the general case of $m$-linear forms $T:X_{p_{1}}^{N}\times\cdots\times
X_{p_{m}}^{N}\rightarrow\mathbb{C}$, let $A$ be a non-void subset of $\left\{
1,...,m\right\}  $ and assume that $p_{j}=\infty$ for every $j\in A$. If $A$
is properly contained in $\{1,\ldots,m\}$ we define the norm
\[
\Vert T\Vert_{A,M}:=\sup\left\{  |T(x^{\left(  1\right)  },...,x^{\left(
m\right)  })|:x^{\left(  i\right)  }\in B_{\ell_{p_{i}}^{N}}\text{ and
}x^{\left(  j\right)  }\in T_{M}^{N}\text{, for all }i\in A^{\complement
}\text{ and }j\in A\right\}  ,
\]
where $A^{\complement}$ denotes the complement of $A$ in $\left\{
1,...,m\right\}  $. We observe that $\Vert T\Vert_{\left\{  1,...,m\right\}
,M}=\Vert T\Vert_{M}$.

Proposition \ref{estnorm} can be extended as follows:

\begin{lemma}
[Basic Lemma]\label{664}Let $N,M,m\in\mathbb{N}$, $M\geq3$, and $p_{1}%
,...,p_{m}\in\lbrack1,\infty]$. If $A$ is a non-void subset of $\left\{
1,...,m\right\}  $ and $p_{j}=\infty$, for every $j\in A$, then
\[
\Vert T\Vert_{A,M}\leq\Vert T\Vert\leq r_{M}^{-\left\vert A\right\vert }\Vert
T\Vert_{A,M},
\]
for all $m$-linear form $T:X_{p_{1}}^{N}\times\cdots\times X_{p_{m}}%
^{N}\rightarrow\mathbb{C}$, where $\left\vert A\right\vert $ denotes the
cardinality of the set $A$ and $r_{M}$ is as in Lemma \ref{apothem}.
\end{lemma}

\section{Multiple Khinchine inequality for Steinhaus variables: best
constants}

The Khinchine inequality was proved in 1923 by A. Khinchine (\cite{kh}) to
deal with the notion of certain random walks as we illustrate with the
following example: suppose that we have $n$ real numbers $a_{1},...,a_{n}$ and
a fair coin. When we flip the coin, if it comes up heads, you choose
$\alpha_{1}=a_{1}$, and if it comes up tails, you choose $\alpha_{1}=-a_{1}.$
Repeating the process, after having flipped the coin $k$ times we have
\[
\alpha_{k+1}:=\alpha_{k}+a_{k+1},
\]
if it comes up heads and
\[
\alpha_{k+1}:=\alpha_{k}-a_{k+1},
\]
if it comes up tails. After $n$ steps, what should be the expected value of
\[
|\alpha_{n}|=\left\vert \sum_{k=1}^{n}\pm a_{k}\right\vert ?
\]
Khinchine's inequality shows that the \textquotedblleft
average\textquotedblright\ $\left(  \frac{1}{2^{n}}\left\vert \sum
\limits_{\varepsilon_{1},...,\varepsilon_{n}=1,-1}\varepsilon_{j}%
a_{j}\right\vert ^{p}\right)  ^{\frac{1}{p}}$ is in some sense equivalent to
the $l_{2}$-norm of $\left(  a_{n}\right)  .$ More precisely, it asserts that
for any $p>0$ there are constants $A_{p},B_{p}>0$ such that
\begin{equation}
A_{p}\left(  \sum\limits_{j=1}^{n}\left\vert a_{j}\right\vert ^{2}\right)
^{\frac{1}{2}}\leq\left(  \frac{1}{2^{n}}\left\vert \sum\limits_{\varepsilon
_{1},...,\varepsilon_{n}=1,-1}\varepsilon_{j}a_{j}\right\vert ^{p}\right)
^{\frac{1}{p}}\leq B_{p}\left(  \sum\limits_{j=1}^{n}\left\vert a_{j}%
\right\vert ^{2}\right)  ^{\frac{1}{2}} \label{65}%
\end{equation}
for all sequence of scalars $\left(  a_{i}\right)  _{i=1}^{n}$ and all
positive integers $n.$ The Khinchine inequality can be rewritten using the
Rademacher variables
\[
r_{n}(t):=sign\left(  \sin2^{n}\pi t\right)
\]
as follows:
\[
A_{p}\left(  \sum_{n=1}^{N}\left\vert a_{n}\right\vert ^{2}\right)  ^{\frac
{1}{2}}\leq\left(  \int\limits_{0}^{1}\left\vert \sum_{n=1}^{N}a_{n}%
r_{n}\left(  t\right)  \right\vert ^{p}dt\right)  ^{\frac{1}{p}}\leq
B_{p}\left(  \sum_{n=1}^{N}\left\vert a_{n}\right\vert ^{2}\right)  ^{\frac
{1}{2}}%
\]
for any linear combination (real or complex) for every positive integer $N$.

Obviously, $A_{p}=1$ for all $p\geq2$ and $B_{p}=1$ for all $p\leq2$. Steckin
\cite[1961]{ste}, Young \cite[1976]{you}, Szarek \cite[1976]{sza} contributed
to the problem of finding the best constants $A_{p}$ and $B_{p}$, for some
non-trivial values of $p$. Finally in 1982 Haagerup \cite{Ha} solved the
problem completely, using techniques of analytic probability.

Haagerup (\cite{Ha}) proved that
\[
A_{p}=\sqrt{2}\left(  \frac{\Gamma\left(  \frac{p+1}{2}\right)  }{\sqrt{\pi}%
}\right)  ^{\frac{1}{p}},\ \ \text{ for }1.85\approx p_{0}<p<2
\]
and
\[
A_{p}=2^{\frac{1}{2}-\frac{1}{p}},\ \ \text{ for }0<p\leq p_{0}\approx1.85.
\]
Above and henceforth $\Gamma$ denotes the famous Gamma function. The exact
definition of the critical value $p_{0}$ is the following: $p_{0}\in(1,2)$ is
the unique real number satisfying%

\begin{equation}
\Gamma\left(  \frac{p_{0}+1}{2}\right)  =\frac{\sqrt{\pi}}{2}. \label{pezero}%
\end{equation}
On the other hand,
\[
B_{p}=\sqrt{2}\left(  \frac{\Gamma\left(  \frac{p+1}{2}\right)  }{\sqrt{\pi}%
}\right)  ^{\frac{1}{p}},\ \ \text{ }p>2.
\]
The counterpart for the average $\left(  \frac{1}{2^{n}}\left\vert
\sum\limits_{\varepsilon_{1},...,\varepsilon_{n}=1,-1}\varepsilon_{j}%
a_{j}\right\vert ^{p}\right)  ^{\frac{1}{p}}$ in the complex framework is
\begin{equation}
\left(  \frac{1}{2\pi}\right)  ^{N}\int_{0}^{2\pi}\ldots\int_{0}^{2\pi
}\left\vert \sum_{n=1}^{N}a_{n}e^{it_{n}}\right\vert ^{p}dt_{1}\cdots dt_{N}.
\label{222}%
\end{equation}
For the sake of simplicity we shall denote (\ref{222}) by
\begin{equation}
\mathbb{E}\left\vert \sum_{n=1}^{N}a_{n}\varepsilon_{n}\right\vert
^{p}\nonumber
\end{equation}
where $\varepsilon_{n}$ are Steinhaus variables; i.e. variables which are
uniformly distributed on the circle $S^{1}$. The following version of the
Khinchine inequality holds and in this case the Khinchine inequality is known
as the Khinchine inequality for Steinhaus variables:

\begin{theorem}
[Khinchine's inequality for Steinhaus variables]\label{stein} For every
$0<p<\infty$, there exist constants $\widetilde{A_{p}}$ and $\widetilde{B_{p}%
}$ such that
\begin{equation}
\widetilde{A_{p}}\left(  \sum_{n=1}^{N}\left\vert a_{n}\right\vert
^{2}\right)  ^{\frac{1}{2}}\leq\left(  \mathbb{E}\left\vert \sum_{n=1}%
^{N}a_{n}\varepsilon_{n}\right\vert ^{p}\right)  ^{\frac{1}{p}}\leq
\widetilde{B_{p}}\left(  \sum_{n=1}^{N}\left\vert a_{n}\right\vert
^{2}\right)  ^{\frac{1}{2}} \label{est}%
\end{equation}
for every positive integer $N$ and scalars $a_{1},\ldots,a_{N}$, where
$\varepsilon_{n}$ are Steinhaus variables.
\end{theorem}

Sawa \cite[1985]{saw}, K\"{o}nig and Kwapie\'{n} \cite[2001]{kk}, and
Baernstein with Culverhouse \cite[2002]{BC}, contributed to the problem of
finding the best constants $\widetilde{A_{p}}$ and $\widetilde{B_{p}}$, for
several values of $p$. Recently, in 2014 K\"{o}nig \cite{kon} finally solved
the problem for all $0<p<\infty$.

In this case, the optimal estimates for $\widetilde{A_{p}}$ and $\widetilde
{B_{p}}$, $0<p<\infty$, are given by:%
\begin{align*}
\widetilde{A_{p}}  &  :=\min\left\{  \left(  \Gamma\left(  \frac{p+2}%
{2}\right)  \right)  ^{\frac{1}{p}},\sqrt{2}\left(  \frac{\Gamma\left(
\frac{p+1}{2}\right)  }{\Gamma\left(  \frac{p+2}{2}\right)  \sqrt{\pi}%
}\right)  ^{\frac{1}{p}},1\right\} \\
\widetilde{B_{p}}  &  :=\max\left\{  \left(  \Gamma\left(  \frac{p+2}%
{2}\right)  \right)  ^{\frac{1}{p}},1\right\}  .
\end{align*}
The multiple Khinchine inequality is a natural and useful extension of the
Khinchine inequality:

\begin{theorem}
[Multiple Khinchine inequality]\label{multikhin}Let $0<p<\infty$,
$m\in\mathbb{N}$, and $\left(  y_{i_{1}...i_{m}}\right)  _{i_{1,}...,i_{m}%
=1}^{N}$ be an array of scalars. There are constants $J_{m,p},~K_{m,p}\geq1$,
such that%
\begin{align}
J_{m,p}\left(  \sum_{i_{1},...,i_{m}=1}^{N}\left\vert y_{i_{1}...i_{m}%
}\right\vert ^{2}\right)  ^{\frac{1}{2}}  &  \leq\left(  \int\limits_{I^{m}%
}\left\vert \sum_{i_{1},...,i_{m}=1}^{N}r_{i_{1}}\left(  t_{1}\right)
...r_{i_{m}}\left(  t_{m}\right)  y_{i_{1}...i_{m}}\right\vert ^{p}%
dt_{1}...dt_{m}\right)  ^{\frac{1}{p}}\label{ppp111}\\
&  \leq K_{m,p}\left(  \sum_{i_{1},...,i_{m}=1}^{N}\left\vert y_{i_{1}%
...i_{m}}\right\vert ^{2}\right)  ^{\frac{1}{2}}\nonumber
\end{align}
for all $N\in\mathbb{N}$, where $r_{i_{j}}\left(  t_{j}\right)  $ are the
Rademacher functions, for all $j\in\left\{  1,...,m\right\}  $ and $i_{j}%
\in\left\{  1,...,N\right\}  .$
\end{theorem}

The final solution giving the optimal constant $J_{m,p}$ was obtained in
\cite[2019]{ns}:%
\begin{equation}
J_{m,p}=\left(  A_{p}\right)  ^{m} \label{nuse}%
\end{equation}
for all $m\in\mathbb{N}$ and for all $0<p<\infty$, where $A_{p}$ is given in
(\ref{65}). We are interested in the version of the multiple Khinchine
inequality for Steinhaus variables. Again, for the sake of simplicity, we
write
\begin{align}
&  \mathbb{E}_{m}\left\vert \sum_{n_{1},\ldots,n_{m}=1}^{N}a_{n_{1}\ldots
n_{m}}\varepsilon_{n_{1}}^{(1)}\cdots\varepsilon_{n_{m}}^{(m)}\right\vert
^{p}=\label{multiple}\\
&  \left(  \frac{1}{2\pi}\right)  ^{Nm}\int_{0}^{2\pi}\ldots\int_{0}^{2\pi
}\left\vert \sum_{n_{1},\ldots,n_{m}=1}^{N}a_{n_{1}\ldots n_{m}}e^{it_{n_{1}%
}^{\left(  1\right)  }}\cdot\ldots\cdot e^{it_{n_{m}}^{\left(  m\right)  }%
}\right\vert ^{p}dt_{n_{1}}^{\left(  1\right)  }\ldots dt_{n_{m}}^{\left(
m\right)  },\nonumber
\end{align}
and the multiple Khinchine inequality reads as follows:

\begin{theorem}
[Multiple Khinchine inequality for Steinhaus variables]\label{d} Let
$0<p<\infty$, $m\geq1$, and let $(a_{i_{1},\ldots,i_{m}})_{i_{1},\ldots
,i_{m}=1}^{N}$ be an array of scalars. There are constants $R_{m,p}%
,S_{m,p}\geq1$, such that
\begin{equation}
S_{m,p}\left(  \sum\limits_{i_{1}\ldots i_{m}=1}^{N}\left\vert a_{i_{1}%
,\ldots,i_{m}}\right\vert ^{2}\right)  ^{1/2}\leq\left(  \mathbb{E}%
_{m}\left\vert \sum_{i_{1},\ldots,i_{m}=1}^{N}a_{i_{1}\ldots i_{m}}%
\varepsilon_{i_{1}}^{(1)}\cdots\varepsilon_{i_{m}}^{(m)}\right\vert
^{p}\right)  ^{\frac{1}{p}}\leq R_{m,p}\left(  \sum\limits_{i_{1},\ldots
,i_{m}=1}^{N}\left\vert a_{i_{1}\ldots i_{m}}\right\vert ^{2}\right)  ^{1/2},
\label{pp}%
\end{equation}
for all $N\in\mathbb{N}$, all $j\in\left\{  1,...,m\right\}  $ and $i_{j}%
\in\left\{  1,...,N\right\}  $. Moreover, $S_{m,p}\geq\left(  \widetilde
{A_{p}}\right)  ^{m}$ and $R_{m,p}\leq\left(  \widetilde{B_{p}}\right)  ^{m}$.
\end{theorem}

The above inequalities are folklore. The inequality in the left hand plays a
crucial role to improve the estimates for the constants in the
Bohnenblust--Hille inequality for complex scalars (see \cite{bayart}). For the
sake of completeness, we give an elementary proof:

\begin{proof}
The case $m=1$ is exactly the Khinchine inequality for Steinhaus variables.
Let us start the proof, by induction, in a first time for the inequality in
the left hand and the case $0<p\leq2$. Assume inductively the result holds for
$m-1$, then%

\begin{align}
&  \left(  \widetilde{A_{p}}\right)  ^{\left(  m-1\right)  }\left(
\sum_{i_{1},...,i_{m}=1}^{N}\left\vert a_{i_{1}...i_{m}}\right\vert
^{2}\right)  ^{\frac{1}{2}}\nonumber\\
&  =\left(  \widetilde{A_{p}}\right)  ^{\left(  m-1\right)  }\left(
\sum_{i_{1}=1}^{N}\left(  \sum_{i_{2},...,i_{m}=1}^{N}\left\vert
a_{i_{1}...i_{m}}\right\vert ^{2}\right)  ^{\frac{1}{2}2}\right)  ^{\frac
{1}{2}}\label{mink1}\\
&  \leq\left(  \sum_{i_{1}=1}^{N}\left(  \mathbb{E}_{m-1}\left\vert
\sum_{i_{2},\ldots,i_{m}=1}^{N}a_{i_{1}\ldots i_{m}}\varepsilon_{i_{2}}%
^{(2)}\cdots\varepsilon_{i_{m}}^{(m)}\right\vert ^{p}\right)  ^{\frac{2}{p}%
}\right)  ^{\frac{1}{2}}.\nonumber
\end{align}

Since $\frac{2}{p}\geq1$, using the Minkowski integral inequality we get
\begin{align}
&  \left(  \sum_{i_{1}=1}^{N}\left(  \mathbb{E}_{m-1}\left\vert \sum
_{i_{2},\ldots,i_{m}=1}^{N}a_{i_{1}\ldots i_{m}}\varepsilon_{i_{2}}%
^{(2)}\cdots\varepsilon_{i_{m}}^{(m)}\right\vert ^{p}\right)  ^{\frac{2}{p}%
}\right)  ^{\frac{1}{2}}\label{mink2}\\
&  \leq\left[  \mathbb{E}_{m-1}\left(  \sum_{i_{1}=1}^{N}\left\vert
\sum_{i_{2},\ldots,i_{m}=1}^{N}a_{i_{1}\ldots i_{m}}\varepsilon_{i_{2}}%
^{(2)}\cdots\varepsilon_{i_{m}}^{(m)}\right\vert ^{p\frac{2}{p}}\right)
^{\frac{p}{2}}\right]  ^{\frac{1}{p}}.\nonumber
\end{align}
From (\ref{mink1}) and (\ref{mink2}), we have
\begin{align}
&  \left(  \widetilde{A_{p}}\right)  ^{\left(  m-1\right)  }\left(
\sum_{i_{1}=1}^{N}\left(  \sum_{i_{2},...,i_{m}=1}^{N}\left\vert
a_{i_{1}...i_{m}}\right\vert ^{2}\right)  ^{\frac{1}{2}2}\right)  ^{\frac
{1}{2}}\label{k-1}\\
&  \leq\left[  \mathbb{E}_{m-1}\left(  \left(  \sum_{{i_{1}}=1}^{N}\left\vert
\sum_{i_{2},\ldots,i_{m}=1}^{N}a_{i_{1}\ldots i_{m}}\varepsilon_{i_{2}}%
^{(2)}\cdots\varepsilon_{i_{m}}^{(m)}\right\vert ^{2}\right)  ^{\frac{1}{2}%
}\right)  ^{p}\right]  ^{\frac{1}{p}}.\nonumber
\end{align}
Now, applying the Khinchine inequality we get
\begin{align}
&  \widetilde{A_{p}}\left(  \sum_{i_{1}=1}^{N}\left\vert \sum_{i_{2}%
,\ldots,i_{m}=1}^{N}a_{i_{1}\ldots i_{m}}\varepsilon_{i_{2}}^{(2)}%
\cdots\varepsilon_{i_{m}}^{(m)}\right\vert ^{2}\right)  ^{\frac{1}{2}%
}\label{um}\\
&  \leq\left(  \mathbb{E}_{1}\left\vert \sum_{i_{1},...,i_{m}=1}^{N}%
a_{i_{1}\ldots i_{m}}\varepsilon_{i_{1}}^{(1)}\cdots\varepsilon_{i_{m}}%
^{(m)}\right\vert ^{p}\right)  ^{\frac{1}{p}},\nonumber
\end{align}
thus, by (\ref{k-1}) and (\ref{um})
\begin{align*}
&  \widetilde{A_{p}}\left(  \widetilde{A_{p}}\right)  ^{\left(  m-1\right)
}\left(  \sum_{i_{1}=1}^{N}\left(  \sum_{i_{2},...,i_{m}=1}^{N}\left\vert
a_{i_{1}...i_{m}}\right\vert ^{2}\right)  ^{\frac{1}{2}2}\right)  ^{\frac
{1}{2}}\\
&  =\widetilde{A_{p}}(\widetilde{A_{p}})^{\left(  m-1\right)  }\left(
\sum_{i_{1}=1}^{N}\left\vert \sum_{i_{2},\ldots,i_{m}=1}^{N}a_{i_{1}\ldots
i_{m}}\varepsilon_{i_{2}}^{(2)}\cdots\varepsilon_{i_{m}}^{(m)}\right\vert
^{2}\right)  ^{\frac{1}{2}}\\
&  \leq\widetilde{A_{p}}\left[  \mathbb{E}_{m-1}\left(  \left(  \sum_{{i_{1}%
}=1}^{N}\left\vert \sum_{i_{2},\ldots,i_{m}=1}^{N}a_{i_{1}\ldots i_{m}%
}\varepsilon_{i_{2}}^{(2)}\cdots\varepsilon_{i_{m}}^{(m)}\right\vert
^{2}\right)  ^{\frac{1}{2}}\right)  ^{p}\right]  ^{\frac{1}{p}}\\
&  =\left[  \mathbb{E}_{m-1}\left(  \widetilde{A_{p}}\left(  \sum_{{i_{1}}%
=1}^{N}\left\vert \sum_{i_{2},\ldots,i_{m}=1}^{N}a_{i_{1}\ldots i_{m}%
}\varepsilon_{i_{2}}^{(2)}\cdots\varepsilon_{i_{m}}^{(m)}\right\vert
^{2}\right)  ^{\frac{1}{2}}\right)  ^{p}\right]  ^{\frac{1}{p}}\\
&  \leq\left[  \mathbb{E}_{m-1}\left(  \mathbb{E}_{1}\left\vert \sum
_{i_{1},...,i_{m}=1}^{N}a_{i_{1}\ldots i_{m}}\varepsilon_{i_{1}}^{(1)}%
\cdots\varepsilon_{i_{m}}^{(m)}\right\vert ^{p}\right)  ^{\frac{p}{p}}\right]
^{\frac{1}{p}}%
\end{align*}
and by Fubini's theorem
\[
(\widetilde{A_{p}})^{m}\left(  \sum_{i_{1},...,i_{m}=1}^{N}\left\vert
a_{i_{1}...i_{m}}\right\vert ^{2}\right)  ^{\frac{1}{2}}\leq\left(
\mathbb{E}_{m}\left\vert \sum_{i_{1},...,i_{m}=1}^{N}a_{i_{1}\ldots i_{m}%
}\varepsilon_{i_{1}}^{(1)}\cdots\varepsilon_{i_{m}}^{(m)}\right\vert
^{p}\right)  ^{\frac{1}{p}}.
\]
We now proceed with the proof of the right hand inequality. We first consider
the case $2\leq p$. One more time, assume inductively that the result holds
for $m-1$, then%
\begin{align}
&  \left(  \sum_{i_{1}=1}^{N}\left(  \mathbb{E}_{m-1}\left\vert \sum
_{i_{2},\ldots,i_{m}=1}^{N}a_{i_{1}\ldots i_{m}}\varepsilon_{i_{2}}%
^{(2)}\cdots\varepsilon_{i_{m}}^{(m)}\right\vert ^{p}\right)  ^{\frac{2}{p}%
}\right)  ^{\frac{1}{2}}\label{mink0}\\
&  \leq\left(  \widetilde{B_{p}}\right)  ^{\left(  m-1\right)  }\left(
\sum_{i_{1}=1}^{N}\left(  \sum_{i_{2},...,i_{m}=1}^{N}\left\vert
a_{i_{1}...i_{m}}\right\vert ^{2}\right)  ^{\frac{1}{2}2}\right)  ^{\frac
{1}{2}}\nonumber\\
&  =\left(  \widetilde{B_{p}}\right)  ^{\left(  m-1\right)  }\left(
\sum_{i_{1},...,i_{m}=1}^{N}\left\vert a_{i_{1}...i_{m}}\right\vert
^{2}\right)  ^{\frac{1}{2}}\nonumber
\end{align}

Since $\frac{2}{p}\leq1$, using the Minkowski integral inequality we get
\begin{align}
&  \left[  \mathbb{E}_{m-1}\left(  \sum_{i_{1}=1}^{N}\left\vert \sum
_{i_{2},\ldots,i_{m}=1}^{N}a_{i_{1}\ldots i_{m}}\varepsilon_{i_{2}}%
^{(2)}\cdots\varepsilon_{i_{m}}^{(m)}\right\vert ^{p\frac{2}{p}}\right)
^{\frac{p}{2}}\right]  ^{\frac{1}{p}}\label{mink3}\\
&  \leq\left(  \sum_{i_{1}=1}^{N}\left(  \mathbb{E}_{m-1}\left\vert
\sum_{i_{2},\ldots,i_{m}=1}^{N}a_{i_{1}\ldots i_{m}}\varepsilon_{i_{2}}%
^{(2)}\cdots\varepsilon_{i_{m}}^{(m)}\right\vert ^{p}\right)  ^{\frac{2}{p}%
}\right)  ^{\frac{1}{2}}.\nonumber
\end{align}
From (\ref{mink0}) and (\ref{mink3}), we have
\begin{align}
&  \left[  \mathbb{E}_{m-1}\left(  \sum_{i_{1}=1}^{N}\left\vert \sum
_{i_{2},\ldots,i_{m}=1}^{N}a_{i_{1}\ldots i_{m}}\varepsilon_{i_{2}}%
^{(2)}\cdots\varepsilon_{i_{m}}^{(m)}\right\vert ^{p\frac{2}{p}}\right)
^{\frac{p}{2}}\right]  ^{\frac{1}{p}}\label{caum}\\
&  \leq\left(  \widetilde{B_{p}}\right)  ^{\left(  m-1\right)  }\left(
\sum_{i_{1},...,i_{m}=1}^{N}\left\vert a_{i_{1}...i_{m}}\right\vert
^{2}\right)  ^{\frac{1}{2}}.\nonumber
\end{align}
Now, applying the Khinchine inequality we get
\begin{align}
&  \left(  \mathbb{E}_{1}\left\vert \sum_{i_{1},...,i_{m}=1}^{N}a_{i_{1}\ldots
i_{m}}\varepsilon_{i_{1}}^{(1)}\cdots\varepsilon_{i_{m}}^{(m)}\right\vert
^{p}\right)  ^{\frac{1}{p}}\label{un}\\
&  \leq\widetilde{B_{p}}\left(  \sum_{i_{1}=1}^{N}\left\vert \sum
_{i_{2},\ldots,i_{m}=1}^{N}a_{i_{1}\ldots i_{m}}\varepsilon_{i_{2}}%
^{(2)}\cdots\varepsilon_{i_{m}}^{(m)}\right\vert ^{2}\right)  ^{\frac{1}{2}%
}\nonumber
\end{align}
thus, from (\ref{caum}) and (\ref{un}), we have
\begin{align*}
&  \left[  \mathbb{E}_{m-1}\left(  \mathbb{E}_{1}\left\vert \sum
_{i_{1},...,i_{m}=1}^{N}a_{i_{1}\ldots i_{m}}\varepsilon_{i_{1}}^{(1)}%
\cdots\varepsilon_{i_{m}}^{(m)}\right\vert ^{p}\right)  ^{\frac{p}{p}}\right]
^{\frac{1}{p}}\\
&  \leq\left[  \mathbb{E}_{m-1}\widetilde{B_{p}}\left(  \sum_{i_{1}=1}%
^{N}\left\vert \sum_{i_{2},\ldots,i_{m}=1}^{N}a_{i_{1}\ldots i_{m}}%
\varepsilon_{i_{2}}^{(2)}\cdots\varepsilon_{i_{m}}^{(m)}\right\vert
^{2}\right)  ^{\frac{p}{2}}\right]  ^{\frac{1}{p}}\\
&  =\widetilde{B_{p}}\left[  \mathbb{E}_{m-1}\left(  \sum_{i_{1}=1}%
^{N}\left\vert \sum_{i_{2},\ldots,i_{m}=1}^{N}a_{i_{1}\ldots i_{m}}%
\varepsilon_{i_{2}}^{(2)}\cdots\varepsilon_{i_{m}}^{(m)}\right\vert
^{2}\right)  ^{\frac{p}{2}}\right]  ^{\frac{1}{p}}\\
&  \leq\widetilde{B_{p}}\left(  \widetilde{B_{p}}\right)  ^{\left(
m-1\right)  }\left(  \sum_{i_{1},...,i_{m}=1}^{N}\left\vert a_{i_{1}...i_{m}%
}\right\vert ^{2}\right)  ^{\frac{1}{2}},
\end{align*}
and by Fubini's theorem
\[
\left[  \mathbb{E}_{m}\left\vert \sum_{i_{1},...,i_{m}=1}^{N}a_{i_{1}\ldots
i_{m}}\varepsilon_{i_{1}}^{(1)}\cdots\varepsilon_{i_{m}}^{(m)}\right\vert
^{p}\right]  ^{\frac{1}{p}}\leq\left(  \widetilde{B_{p}}\right)  ^{m}\left(
\sum_{i_{1},...,i_{m}=1}^{N}\left\vert a_{i_{1}...i_{m}}\right\vert
^{2}\right)  ^{\frac{1}{2}}.
\]
On the other hand, since $\widetilde{B_{2}}=\widetilde{A_{2}}=1,$ we have
particularly proved by induction which
\[
\left(  \sum_{i_{1},...,i_{m}=1}^{N}\left\vert a_{i_{1}...i_{m}}\right\vert
^{2}\right)  ^{\frac{1}{2}}\leq\left(  \mathbb{E}_{m}\left\vert \sum
_{i_{1},...,i_{m}=1}^{N}a_{i_{1}\ldots i_{m}}\varepsilon_{i_{1}}^{(1)}%
\cdots\varepsilon_{i_{m}}^{(m)}\right\vert ^{2}\right)  ^{\frac{1}{2}}%
\leq\left(  \sum_{i_{1},...,i_{m}=1}^{N}\left\vert a_{i_{1}...i_{m}%
}\right\vert ^{2}\right)  ^{\frac{1}{2}},
\]
for all $m\geq1$, the assertions in the left hand for the case $2\leq p$, and
in the right hand for the case $0<p<2$ follow trivially from the norm (metric)
inclusion between the $L_{p}$ sets, i.e., if $0<p<q<\infty$ then
\[
\left(  \mathbb{E}_{m}\left\vert \sum_{i_{1},...,i_{m}=1}^{N}a_{i_{1}\ldots
i_{m}}\varepsilon_{i_{1}}^{(1)}\cdots\varepsilon_{i_{m}}^{(m)}\right\vert
^{p}\right)  ^{\frac{1}{p}}\leq\left(  \mathbb{E}_{m}\left\vert \sum
_{i_{1},...,i_{m}=1}^{N}a_{i_{1}\ldots i_{m}}\varepsilon_{i_{1}}^{(1)}%
\cdots\varepsilon_{i_{m}}^{(m)}\right\vert ^{q}\right)  ^{\frac{1}{q}}.
\]

Thus, the inequalities follow, for all $m\geq1$ and $0<p<\infty,$ with
estimates $S_{m,p}\geq\left(  \widetilde{A_{p}}\right)  ^{m}$ and $R_{m,p}%
\leq\left(  \widetilde{B_{p}}\right)  ^{m}$.
\end{proof}

Our goal in this section is to prove the analogue of the estimate (\ref{nuse})
for the case of Steinhaus variables. We will prove that the optimal constants
$S_{m,p}$ and $R_{m,p}$ are $(\widetilde{A_{p}})^{m}$ and $(\widetilde{B_{p}%
})^{m}$, respectively, for all $m\in\mathbb{N}$ and for all $0<p<\infty.$ We
shall show that the optimal constants can be obtained as a consequence of the
following fundamental result of K\"{o}nig:

\begin{theorem}
\cite[Theorem 1]{kon} \label{zentral} Let $p\in\left(  0,\infty\right)  $.
Then%
\begin{equation}
\left(  \mathbb{E}\left\vert \sum_{i=1}^{2}\varepsilon_{i}\right\vert
^{p}\right)  ^{\frac{1}{p}}=2\left(  \frac{\Gamma\left(  \frac{p+1}{2}\right)
}{\Gamma\left(  \frac{p+2}{2}\right)  \sqrt{\pi}}\right)  ^{\frac{1}{p}}
\label{0001}%
\end{equation}
and
\begin{equation}
\lim_{N\rightarrow\infty}\left(  \mathbb{E}\left\vert \sum_{i=1}^{N}\frac
{1}{\sqrt{N}}\varepsilon_{i}\right\vert ^{p}\right)  ^{\frac{1}{p}}=\left(
\Gamma\left(  \frac{p+2}{2}\right)  \right)  ^{\frac{1}{p}}\text{,}
\label{0004}%
\end{equation}
where $\varepsilon_{i}$ are denoting Steinhaus variables, for all $i$.
Moreover, the optimal constants $\widetilde{A_{p}}$ and\ $\widetilde{B_{p}}$
in the Khinchine inequalities for Steinhaus variables are:%
\begin{equation}
\widetilde{B_{p}}=\left(  \Gamma\left(  \frac{p+2}{2}\right)  \right)
^{\frac{1}{p}},\ \ \text{ for }2\leq p, \label{0008}%
\end{equation}
\begin{equation}
\widetilde{A_{p}}=\left(  \Gamma\left(  \frac{p+2}{2}\right)  \right)
^{\frac{1}{p}},\ \ \text{ for }0.4756\approx p_{1}\leq p<2 \label{0009}%
\end{equation}
and%
\begin{equation}
\widetilde{A_{p}}=\sqrt{2}\left(  \frac{\Gamma\left(  \frac{p+1}{2}\right)
}{\Gamma\left(  \frac{p+2}{2}\right)  \sqrt{\pi}}\right)  ^{\frac{1}{p}%
},\ \ \text{ for }0<p<p_{1}\approx0.4756 \label{0002}%
\end{equation}
where $p_{1}\in(0,1)$ is the unique real number satisfying%
\[
1=\sqrt{2}\left(  \frac{\Gamma\left(  \frac{p_{1}+1}{2}\right)  }{\sqrt{\pi
}\Gamma\left(  \frac{p_{1}}{2}+1\right)  ^{2}}\right)  ^{\frac{1}{p}}.
\]

\end{theorem}

Borrowing ideas from \cite{ns},we can now prove the following:

\begin{proposition}
\label{e}For all $0<p<\infty$ and $m\in\mathbb{N}$, the optimal constants
$S_{m,p}$ , and $R_{m,p}$ in \eqref{pp} are $\left(  \widetilde{A_{p}}\right)
^{m}$, and $\left(  \widetilde{B_{p}}\right)  ^{m}$ respectively.
\end{proposition}

\begin{proof}
From Theorem \ref{multikhin} we already know that $S_{m,p}\geq\left(
\widetilde{A_{p}}\right)  ^{m}$ and $R_{m,p}\leq\left(  \widetilde{B_{p}%
}\right)  ^{m}$ for all $0<p<\infty$ and $m\in\mathbb{N}$, so we only need to
check the other inequalities; i.e. $S_{m,p}\leq\left(  \widetilde{A_{p}%
}\right)  ^{m}$ and $R_{m,p}\geq\left(  \widetilde{B_{p}}\right)  ^{m}$. Let
us start with the estimate concerning the constant $S_{m,p}$.

Obviously $S_{m,p}\leq\left(  \widetilde{A_{p}}\right)  ^{m}=\left(  1\right)
^{m}=1,$ for all $2\leq p<\infty$. Let us prove the case $0<p<2$. We begin by
considering $0<p<p_{1}$, where $p_{1}$ is the number defined in the previous
lemma. Define
\begin{equation}
a_{i_{1}...i_{m}}=%
\begin{cases}
1, & i_{1},\ldots,i_{m}\in\left\{  1,2\right\} \\
0, & i_{l}\in\left\{  3,\ldots,N\right\}  \text{ for some }l.
\end{cases}
\label{menor}%
\end{equation}
Then,
\[
\left(  \sum_{i_{1},...,i_{m}=1}^{N}\left\vert a_{i_{1}...i_{m}}\right\vert
^{2}\right)  ^{\frac{1}{2}}=2^{\frac{m}{2}}.
\]
On the other hand, we have
\[
\left\vert \sum_{i_{1},\ldots,i_{m}=1}^{N}a_{i_{1},\ldots,i_{m}}%
\varepsilon_{i_{1}}^{(1)}\ldots\varepsilon_{i_{m}}^{(m)}\right\vert
^{p}=\left\vert \sum_{i_{1},\ldots,i_{m}=1}^{2}\varepsilon_{i_{1}}^{(1)}%
\ldots\varepsilon_{i_{m}}^{(m)}\right\vert ^{p}=\prod_{j=1}^{m}\left\vert
\sum_{i_{j}=1}^{2}\varepsilon_{i_{j}}^{(j)}\right\vert ^{p}.
\]
Hence
\begin{align*}
&  \left(  \mathbb{E}_{m}\left\vert \sum_{i_{1},\ldots,i_{m}=1}^{N}%
a_{i_{1},\ldots,i_{m}}\varepsilon_{i_{1}}^{(1)}\ldots\varepsilon_{i_{m}}%
^{(m)}\right\vert ^{p}\right)  ^{\frac{1}{p}}\\
&  =\left(  \mathbb{E}_{m}\prod_{j=1}^{m}\left\vert \sum_{i_{j}=1}%
^{2}\varepsilon_{i_{j}}^{(j)}\right\vert ^{p}\right)  ^{\frac{1}{p}}\\
&  =\left(  \prod_{j=1}^{m}\mathbb{E}_{m}\left\vert \sum_{i_{j}=1}%
^{2}\varepsilon_{i_{j}}^{(j)}\right\vert ^{p}\right)  ^{\frac{1}{p}}\\
&  =\left(  \prod_{j=1}^{m}\mathbb{E}\left\vert \sum_{i_{j}=1}^{2}%
\varepsilon_{i_{j}}^{(j)}\right\vert ^{p}\right)  ^{\frac{1}{p}}\\
&  =\left[  \mathbb{E}\left\vert \sum_{i=1}^{2}\varepsilon_{i}\right\vert
^{p}\right]  ^{\frac{m}{p}}.
\end{align*}
The previous equality combined with (\ref{0001}) and (\ref{0002}) provides%
\[
\left(  \mathbb{E}_{m}\left\vert \sum_{i_{1},\ldots,i_{m}=1}^{N}%
a_{i_{1},\ldots,i_{m}}\varepsilon_{i_{1}}^{(1)}\ldots\varepsilon_{i_{m}}%
^{(m)}\right\vert ^{p}\right)  ^{\frac{1}{p}}=\left[  \mathbb{E}\left\vert
\sum_{i=1}^{2}\varepsilon_{i}\right\vert ^{p}\right]  ^{\frac{m}{p}}=\left(
2^{\frac{1}{2}}\widetilde{A_{p}}\right)  ^{m}.
\]
Hence
\[
S_{m,r}\leq\frac{\left(  2^{\frac{1}{2}}\widetilde{A_{p}}\right)  ^{m}%
}{2^{\frac{m}{2}}}=\left(  \widetilde{A_{p}}\right)  ^{m}.
\]
i.e., the assertion is proved for $0<p<p_{1}$ and $m\in\mathbb{N}$.

On the other hand, let $p\in\lbrack p_{1},2)$ and $m\in\mathbb{N}$. For each
positive integer $N$, consider
\[
a_{i_{1}...i_{m}}=\frac{1}{N^{\frac{m}{2}}},\ \ \ \mbox{ for all }i_{1}%
,\ldots,i_{m}\in\left\{  1,N\right\}  .
\]
We have
\[
\left(  \sum_{i_{1},...,i_{m}=1}^{N}\left\vert a_{i_{1}...i_{m}}\right\vert
^{2}\right)  ^{\frac{1}{2}}=\left(  \sum_{i_{1},...,i_{m}=1}^{N}\frac{1}%
{N^{m}}\right)  ^{\frac{1}{2}}=1,
\]
and, at the same time, we have
\[
\left\vert \sum_{i_{1},...,i_{m}=1}^{N}a_{i_{1},\ldots,i_{m}}\varepsilon
_{i_{1}}^{(1)}\ldots\varepsilon_{i_{m}}^{(m)}\right\vert ^{p}=\prod_{j=1}%
^{m}\left\vert \sum_{i_{j}=1}^{N}\frac{1}{\sqrt{N}}\varepsilon_{i_{j}}%
^{(j)}\right\vert ^{p}.
\]
Then,
\begin{align*}
&  \left(  \mathbb{E}_{m}\left\vert \sum_{i_{1},...,i_{m}=1}^{N}%
a_{i_{1},\ldots,i_{m}}\varepsilon_{i_{1}}^{(1)}\ldots\varepsilon_{i_{m}}%
^{(m)}\right\vert ^{p}\right)  ^{\frac{1}{p}}\\
&  =\left(  \mathbb{E}_{m}\prod_{j=1}^{m}\left\vert \sum_{i_{j}=1}^{N}\frac
{1}{\sqrt{N}}\varepsilon_{i_{j}}^{(j)}\right\vert ^{p}\right)  ^{\frac{1}{p}%
}\\
&  =\left(  \prod_{j=1}^{m}\mathbb{E}_{m}\left\vert \sum_{i_{j}=1}^{N}\frac
{1}{\sqrt{N}}\varepsilon_{i_{j}}^{(j)}\right\vert ^{p}\right)  ^{\frac{1}{p}%
}\\
&  =\left(  \prod_{j=1}^{m}\mathbb{E}\left\vert \sum_{i_{j}=1}^{N}\frac
{1}{\sqrt{N}}\varepsilon_{i_{j}}^{(j)}\right\vert ^{p}\right)  ^{\frac{1}{p}%
}\\
&  =\left(  \mathbb{E}\left\vert \sum_{i=1}^{N}\frac{1}{\sqrt{N}}%
\varepsilon_{i}\right\vert ^{p}\right)  ^{\frac{m}{p}}.
\end{align*}
Thus
\[
S_{m,p}\leq\left(  \mathbb{E}\left\vert \sum_{i=1}^{N}\frac{1}{\sqrt{N}%
}\varepsilon_{i}\right\vert ^{p}\right)  ^{\frac{m}{p}}%
\]
By (\ref{0004}) and (\ref{0009}), letting $N\rightarrow\infty$ we obtain
\[
S_{m,p}\leq\left(  \widetilde{A_{p}}\right)  ^{m}%
\]
for $p\in(p_{1},2)$ and $m\in\mathbb{N}$.

Now, we deal with the respective inequality for the constant $R_{m,p}$.
Obviously $R_{m,p}\leq\left(  \widetilde{B_{p}}\right)  ^{m}=\left(  1\right)
^{m}=1,$ for all $0<p\leq2$. Let us prove the case $2<p$. Let $p\in(2,\infty)$
and $m\in\mathbb{N}$. For each positive integer $N$, we again consider%
\[
a_{i_{1}...i_{m}}=\frac{1}{N^{\frac{m}{2}}},\ \ \ \mbox{ for all }i_{1}%
,\ldots,i_{m}\in\left\{  1,N\right\}  .
\]
We already have
\[
\left(  \sum_{i_{1},...,i_{m}=1}^{N}\left\vert a_{i_{1}...i_{m}}\right\vert
^{2}\right)  ^{\frac{1}{2}}=\left(  \sum_{i_{1},...,i_{m}=1}^{N}\frac{1}%
{N^{m}}\right)  ^{\frac{1}{2}}=1,
\]
and,
\[
\left(  \mathbb{E}_{m}\left\vert \sum_{i_{1},...,i_{m}=1}^{N}a_{i_{1}%
,\ldots,i_{m}}\varepsilon_{i_{1}}^{(1)}\ldots\varepsilon_{i_{m}}%
^{(m)}\right\vert ^{p}\right)  ^{\frac{1}{p}}=\left(  \mathbb{E}\left\vert
\sum_{i=1}^{N}\frac{1}{\sqrt{N}}\varepsilon_{i}\right\vert ^{p}\right)
^{\frac{m}{p}}.
\]
Thus
\[
R_{m,p}\geq\left(  \mathbb{E}\left\vert \sum_{i=1}^{N}\frac{1}{\sqrt{N}%
}\varepsilon_{i}\right\vert ^{p}\right)  ^{\frac{m}{p}}.
\]
By (\ref{0004}) and (\ref{0008}), letting $N\rightarrow\infty$ we obtain
\[
R_{m,p}\geq\left(  \widetilde{B_{p}}\right)  ^{m}%
\]
for $p\in(2,\infty)$ and $m\in\mathbb{N}$.
\end{proof}

\begin{remark}
Combining the ideas of this section and those from \cite{ns} we can easily
prove that the optimal constants of the right-hand-side inequality from
(\ref{ppp111}) are also
\[
K_{m,p}=\left(  B_{p}\right)  ^{m}%
\]
for all $m\in\mathbb{N}$ and $0<p<\infty$, where $B_{p}$ is as in (\ref{65}).
\end{remark}

\section{Best constants for mixed-type Littlewood inequalities}

If we look for a common thread in the "different" proofs of the
Hardy--Littlewood inequalities (\ref{q3}), we necessarily find the following
inequality, that we may call mixed Littlewood inequality:

\begin{theorem}
\cite[Theorem A]{pra}\label{pra} Assume $1\leq p_{1},\ldots,p_{m}\leq\infty$
are such that $\frac{1}{p_{1}}+\cdots+\frac{1}{p_{m}}<\frac{1}{2}$ and let
$\frac{1}{\lambda}=1-\left(  \frac{1}{p_{1}}+\cdots+\frac{1}{p_{m}}\right)  $.
Then for every continuous $m$-linear form $T:X_{p_{1}}\times\cdots\times
X_{p_{m}}\rightarrow\mathbb{K}$ we have
\[
\left(  \sum_{i_{1}=1}^{\infty}\left(  \sum_{i_{2},...,i_{m}=1}^{\infty
}\left\vert T(e_{i_{1}},\ldots,e_{i_{m}})\right\vert ^{2}\right)  ^{\lambda
/2}\right)  ^{1/\lambda}\leq\left(  \sqrt{2}\right)  ^{m-1}\Vert T\Vert\,.
\]

\end{theorem}

This inequality appeared for the first time for bilinear forms in 1930
\cite[Theorem 1]{Lit}, with $X_{p_{1}}=X_{p_{2}}=c_{0}$, and then in 1934
\cite[Theorem 1]{Hardy}, 1981 \cite[Theorem A]{pra}, 2016 \cite[Proposition
3.1]{dima}, 2016 \cite[Lemma 2.1]{abps}. The role of this mixed inequality, in
the proofs of the Hardy--Littlewood inequalities in the above references, is
essentially the same (this is described in Bayart's paper \cite{bayart2} in
what he calls Abstract Hardy--Littlewood Method). In fact, in these
references, the mixed inequality always was used as the starting point of the
proof of the Hardy--Littlewood inequality.

We recall the following particular case of the Hardy--Littlewood inequalities
(\ref{q3}). For $p\in\lbrack2,\infty]$ and $m\in\mathbb{N}$, $m\geq2$, the
mixed Littlewood-type inequality asserts that there is an optimal constant
$C_{(p^{\ast},2,..,2)}^{id,\left(  p,\infty,...,\infty\right)  \mathbb{K}}%
\geq1$ such that
\begin{equation}
\left(  \sum_{i_{1}=1}^{\infty}\left(  \sum_{i_{2},...,i_{m}=1}^{\infty
}|T(e_{i_{1}},...,e_{i_{m}})|^{2}\right)  ^{\frac{1}{2}p^{\ast}}\right)
^{\frac{1}{p^{\ast}}}\leq C_{(p^{\ast},2,..,2)}^{id,\left(  p,\infty
,...,\infty\right)  \mathbb{K}}\Vert T\Vert\label{9u}%
\end{equation}
for all continuous $m$-linear forms $T:X_{p}\times X_{\infty}\times\dots\times
X_{\infty}\rightarrow\mathbb{K}$. We recall that according to Theorem
\ref{pra} $C_{(m),p}^{\mathbb{K}}\leq\left(  \sqrt{2}\right)  ^{m-1}.$ In the
recent years, several authors (\cite{pell, racsam, ns, pt}) were working on
estimating the optimal constants in (\ref{9u}) and managed to solve the
problem for the case $\mathbb{K=R}$. In the case of complex scalars, despite
the results achieved in the real case, the optimal constant for all values of
$p$ is unknown. Using a famous inequality due to Minkowski (see \cite{garling}%
) it is simple to verify that the constant $C_{(p^{\ast},2,..,2)}^{id,\left(
p,\infty,...,\infty\right)  \mathbb{K}}$ also dominates the optimal constants
in the following inequalities:%
\begin{align}
\left(  \sum_{i_{2}=1}^{\infty}\left(  \sum_{i_{1}=1}^{\infty}\left(
\sum_{i_{3},...,i_{m}=1}^{\infty}|T(e_{i_{1}},...,e_{i_{m}})|^{2}\right)
^{\frac{1}{2}p^{\ast}}\right)  ^{\frac{1}{p^{\ast}}2}\right)  ^{1/2}  &  \leq
C_{(2,p^{\ast},2,..,2)}^{\sigma_{2},\left(  p,\infty,...,\infty\right)
\mathbb{K}}\Vert T\Vert\label{orlicz}\\
&  \vdots\nonumber\\
\left(  \sum_{i_{2},...,i_{m}=1}^{\infty}\left(  \sum_{i_{1}=1}^{\infty
}|T(e_{i_{1}},...,e_{i_{m}})|^{p^{\ast}}\right)  ^{\frac{1}{p^{\ast}}%
2}\right)  ^{1/2}  &  \leq C_{(2,...,2,p^{\ast})}^{\sigma_{m},\left(
p,\infty,...,\infty\right)  \mathbb{K}}\Vert T\Vert,\nonumber
\end{align}
where $\sigma_{j}$ are the identity maps, except for the case $\sigma
_{j}(j)=1$ and $\sigma_{j}(1)=j.$ We thus have%
\begin{equation}
C_{(2,...,2,p^{\ast})}^{\sigma_{m},\left(  p,\infty,...,\infty\right)
\mathbb{K}}\leq\cdots\leq C_{(2,p^{\ast},2,..,2)}^{\sigma_{2},\left(
p,\infty,...,\infty\right)  \mathbb{K}}\leq C_{(p^{\ast},2,..,2)}^{id,\left(
p,\infty,...,\infty\right)  \mathbb{K}}. \label{ppkk}%
\end{equation}
Now we shall show that, for all $p\in\left[  2,\infty\right]  $ all the
optimal constants are%
\begin{equation}
C_{(2,...,2,p^{\ast})}^{\sigma_{m},\left(  p,\infty,...,\infty\right)
\mathbb{K}}=\cdots=C_{(2,p^{\ast},2,..,2)}^{\sigma_{2},\left(  p,\infty
,...,\infty\right)  \mathbb{K}}=C_{(p^{\ast},2,..,2)}^{id,\left(
p,\infty,...,\infty\right)  \mathbb{K}}=\left(  \widetilde{A_{p^{\ast}}%
}\right)  ^{-\left(  m-1\right)  }, \label{a12}%
\end{equation}
where the notation is as in the Khinchine inequality for Steinhaus variables
(Theorem \ref{stein}). Note that when $p=\infty,$ for reasons of symmetry, we
recover (iii) for all bijections $\sigma$.

We start with the following proposition, showing that $C_{(p^{\ast}%
,2,..,2)}^{id,\left(  p,\infty,...,\infty\right)  \mathbb{C}}\leq\left(
\widetilde{A_{p^{\ast}}}\right)  ^{-\left(  m-1\right)  }$. This estimate is
somewhat new; for real scalars, in \cite[Theorem 2]{racsam} it was proved that
$C_{(p^{\ast},2,..,2)}^{id,\left(  p,\infty,...,\infty\right)  \mathbb{R}}%
\leq\left(  A_{p^{\ast}}\right)  ^{-\left(  m-1\right)  }$ but for complex
scalars the only known estimate was the one given by Theorem \ref{pra}.
However, the proof is simple and follows the lines of the proof of
\cite[Theorem 2]{racsam}.

\begin{proposition}
[Multilinear mixed $\left(  \ell_{\frac{p}{p-1}},\ell_{2}\right)  $-Littlewood
inequality]\label{multiracsam} For all positive integers $m\geq2$ we have
\[
C_{(p^{\ast},2,..,2)}^{id,\left(  p,\infty,...,\infty\right)  \mathbb{R}}%
\leq\left(  S_{m-1,p}\right)  ^{-1}.
\]

\end{proposition}

\begin{proof}
Let $N$ be a positive integer and $T:X_{p}^{N}\times X_{\infty}^{N}\times
\dots\times X_{\infty}^{N}\rightarrow\mathbb{C}$ be a continuous $m$-linear
form. By Theorem \ref{d} we know that%

\begin{align*}
&  \left(  \sum_{i_{1}=1}^{N}\left(  \sum_{i_{2},...,i_{m}=1}^{N}|T(e_{i_{1}%
},...,e_{i_{m}})|^{2}\right)  ^{\frac{1}{2}\times p^{\ast}}\right)  ^{\frac
{1}{p^{\ast}}}\\
&  \quad\leq\left(  S_{m-1,p}\right)  ^{-1}\left(  \sum_{i_{1}=1}%
^{N}\mathbb{E}_{m-1}\left\vert \sum_{i_{2},...,i_{m}}^{N}\varepsilon_{i_{2}%
}^{(2)}\cdots\varepsilon_{i_{m}}^{(m)}T(e_{i_{1}},...,e_{i_{m}})\right\vert
^{p^{\ast}}\right)  ^{\frac{1}{p^{\ast}}}\\
&  =\left(  S_{m-1,p}\right)  ^{-1}\left(  \mathbb{E}_{m-1}\sum_{i_{1}=1}%
^{N}\left\vert T(e_{i_{1}},\sum_{i_{2}}^{N}\varepsilon_{i_{2}}^{(2)}e_{i_{2}%
},...,\sum_{i_{m}}^{N}\varepsilon_{i_{m}}^{(m)}e_{i_{m}})\right\vert
^{p^{\ast}}\right)  ^{\frac{1}{p^{\ast}}}\\
&  \leq\left(  S_{m-1,p}\right)  ^{-1}\left(  \mathbb{E}_{m-1}\left\Vert
T\left(  \cdot,\sum_{i_{2}}^{N}\varepsilon_{i_{2}}^{(2)}e_{i_{2}}%
,...,\sum_{i_{m}}^{N}\varepsilon_{i_{m}}^{(m)}e_{i_{m}}\right)  \right\Vert
^{p^{\ast}}\right)  ^{\frac{1}{p^{\ast}}}\\
&  \quad\leq\left(  S_{m-1,p}\right)  ^{-1}\sup_{\varepsilon_{i_{2}}%
^{(2)},...,\varepsilon_{i_{m}}^{(m)}}\left\Vert T\left(  \cdot,\sum_{i_{2}%
}^{N}\varepsilon_{i_{2}}^{(2)}e_{i_{2}},...,\sum_{i_{m}}^{N}\varepsilon
_{i_{m}}^{(m)}e_{i_{m}}\right)  \right\Vert \\
&  \quad\leq\left(  S_{m-1,p}\right)  ^{-1}\Vert T\Vert.
\end{align*}

\end{proof}

In order to complete the proof of (\ref{a12}) we need some preparatory
results. We need to introduce some notation: For each $M\geq2$, let
$\Omega_{M}$ be the set
\[
\left\{  \frac{2j\pi}{M}:j=0,\ldots,M-1\right\}
\]
Let $m\in\mathbb{N}$, and $(a_{n_{1}\ldots n_{m}})_{n_{1},\ldots,n_{m}=1}^{N}$
be an array of scalars, and $0<p<\infty,$ and $M\geq2$. We define%
\[
E_{m,M,p}\left(  (a_{n_{1},\ldots,n_{m}})_{n_{1},\ldots,n_{m}=1}^{N}\right)
=\left(  \left(  \frac{1}{M}\right)  ^{Nm}\sum_{\left(  t_{n_{1}}%
,...,t_{n_{m}}\right)  \in\left(  \Omega_{M}^{N}\right)  ^{m}}\left\vert
\sum_{n_{1},\ldots,n_{m}=1}^{N}a_{n_{1}\ldots n_{m}}e^{it_{n_{1}}^{\left(
1\right)  }}\cdot\ldots\cdot e^{it_{n_{m}}^{\left(  m\right)  }}\right\vert
^{p}\right)  ^{\frac{1}{p}}.
\]
Using the Dominated Convergence Theorem it is possible to prove that
\[
\lim_{M\rightarrow\infty}\left(  E_{m,M,p}\left(  (a_{i_{1},\ldots,i_{m}%
})_{n_{1},\ldots,n_{m}=1}^{N}\right)  \right)  =\left(  \mathbb{E}%
_{m}\left\vert \sum_{n_{1},\ldots,n_{m}=1}^{N}a_{n_{1}\ldots n_{m}}%
\varepsilon_{n_{1}}^{(1)}\cdot\ldots\cdot\varepsilon_{n_{m}}^{(m)}\right\vert
^{p}\right)  ^{\frac{1}{p}}\text{.}%
\]
We need the following auxiliary result:

\begin{lemma}
Let $m\in\mathbb{N}$, and $(a_{n_{1}\ldots n_{m}})_{n_{1},\ldots,n_{m}=1}^{N}$
be an array of scalars, and $0<p<\infty,$ and $M\geq2$. Then
\[
E_{m,M,p}\left(  (a_{n_{1},\ldots,n_{m}})_{n_{1},\ldots,n_{m}=1}^{N}\right)
=E_{m,M,p}\left(  (a_{n_{1},\ldots,n_{m}}e^{is_{n_{1}}^{\left(  1\right)  }%
}\cdot\ldots\cdot e^{is_{n_{m}}^{\left(  m\right)  }})_{n_{1},\ldots,n_{m}%
=1}^{N}\right)
\]
for all $\mathbf{s}=(s_{n_{1}}^{\left(  1\right)  },...,s_{n_{m}}^{\left(
m\right)  })\in\left(  \Omega_{M}^{N}\right)  ^{m}$.
\end{lemma}

\begin{proof}
If $\mathbf{s}=(s_{n_{1}}^{\left(  1\right)  },...,s_{n_{m}}^{\left(
m\right)  })\in\left(  \Omega_{M}^{N}\right)  ^{m}$, then
\begin{align*}
&  E_{m,M,p}\left(  (a_{n_{1},\ldots,n_{m}}e^{is_{n_{1}}^{\left(  1\right)  }%
}\cdot\ldots\cdot e^{is_{n_{m}}^{\left(  m\right)  }})_{n_{1},\ldots,n_{m}%
=1}^{N}\right) \\
&  =\left(  \left(  \frac{1}{M}\right)  ^{Nm}\sum_{\left(  t_{n_{1}}^{\left(
1\right)  },...,t_{n_{m}}^{\left(  m\right)  }\right)  \in\left(  \Omega
_{M}^{N}\right)  ^{m}}\left\vert \sum_{n_{1},\ldots,n_{m}=1}^{N}a_{n_{1}\ldots
n_{m}}e^{is_{n_{1}}^{\left(  1\right)  }}\cdot\ldots\cdot e^{is_{n_{m}%
}^{\left(  m\right)  }}e^{it_{n_{1}}^{\left(  1\right)  }}\cdot\ldots\cdot
e^{it_{n_{m}}^{\left(  m\right)  }}\right\vert ^{p}\right)  ^{\frac{1}{p}}\\
&  =\left(  \left(  \frac{1}{M}\right)  ^{Nm}\sum_{\left(  t_{n_{1}}^{\left(
1\right)  },...,t_{n_{m}}^{\left(  m\right)  }\right)  \in\left(  \Omega
_{M}^{N}\right)  ^{m}}\left\vert \sum_{n_{1},\ldots,n_{m}=1}^{N}a_{n_{1}\ldots
n_{m}}e^{i\left(  s_{n_{1}}^{\left(  1\right)  }+t_{n_{1}}^{\left(  1\right)
}\right)  }\cdot\ldots\cdot e^{i\left(  s_{n_{m}}^{\left(  m\right)
}+t_{n_{m}}^{\left(  m\right)  }\right)  }\right\vert ^{p}\right)  ^{\frac
{1}{p}}\\
&  =\left(  \left(  \frac{1}{M}\right)  ^{Nm}\sum_{\left(  u_{n_{1}}^{\left(
1\right)  },...,u_{n_{m}}^{\left(  m\right)  }\right)  \in\left(  \Omega
_{M}^{N}\right)  ^{m}}\left\vert \sum_{n_{1},\ldots,n_{m}=1}^{N}a_{n_{1}\ldots
n_{m}}e^{iu_{n_{1}}^{\left(  1\right)  }}\cdot\ldots\cdot e^{iu_{n_{m}%
}^{\left(  m\right)  }}\right\vert ^{p}\right)  ^{\frac{1}{p}}\\
&  =E_{m,M,p}\left(  (a_{n_{1},\ldots,n_{m}})_{n_{1},\ldots,n_{m}=1}%
^{N}\right)  .
\end{align*}

\end{proof}

\bigskip

Now, we enunciate and prove a kind of multiple Khinchine inequality that
extends and unifies the left side of the inequalities in Theorems \ref{multikhin} and  \ref{d}. We emphasize that the  theorem below was
introduced by R. Blei, in the linear case ($m=1$), in \cite[chapter II:
section 6]{blei}.

\begin{theorem}
[Multiple Blei--Khinchine inequality]\label{665}Let $1\leq p\leq2$, and
$m\geq1$, and $M\geq2$, and let $(a_{n_{1},\ldots,n_{m}})_{n_{1},\ldots
,n_{m}=1}^{N}$ be an array of scalars. There is a constant $S_{m,M,p}\geq1$,
such that
\begin{equation}
\left(  \sum\limits_{n_{1},\ldots,n_{m}=1}^{N}\left\vert a_{n_{1}\ldots n_{m}%
}\right\vert ^{2}\right)  ^{1/2}\leq S_{m,M,p}\cdot E_{m,M,p}\left(
(a_{n_{1}\ldots n_{m}})_{n_{1},\ldots,n_{m}=1}^{N}\right)  , \label{unifies}%
\end{equation}
for all $N\in\mathbb{N}$. Moreover,
\[
S_{m,M,p}\leq C_{(2,...,2,p)}^{\sigma_{m+1},\left(  p^{\ast},\infty
,...,\infty\right)  \mathbb{C}}\cdot r_{M}^{-m}.
\]

\end{theorem}

\begin{remark}
 If $M=2,$
the inequality (\ref{unifies}) recovers Theorem \ref{multikhin}, and when
$M\rightarrow\infty,$ the inequality (\ref{unifies}) recovers Theorem \ref{d}.
\end{remark}

\begin{proof}
Let $(a_{n_{1},\ldots,n_{m}})_{n_{1},\ldots,n_{m}=1}^{N}$ be an array of
scalars, such that $E_{m,M,p}\left(  (a_{n_{1},\ldots,n_{m}})_{n_{1}%
,\ldots,n_{m}=1}^{N}\right)  =1$. Then, by the previous lemma,
\begin{align*}
&  E_{m,M,p}\left(  (a_{n_{1},\ldots,n_{m}}e^{is_{n_{1}}^{\left(  1\right)  }%
}\cdot\ldots\cdot e^{is_{n_{m}}^{\left(  m\right)  }})_{n_{1},\ldots,n_{m}%
=1}^{N}\right) \\
&  =\left(  \left(  \frac{1}{M}\right)  ^{Nm}\sum_{\left(  t_{n_{1}}^{\left(
1\right)  },...,t_{n_{m}}^{\left(  m\right)  }\right)  \in\left(  \Omega
_{M}^{N}\right)  ^{m}}\left\vert \sum_{n_{1},\ldots,n_{m}=1}^{N}a_{n_{1}\ldots
n_{m}}e^{is_{n_{1}}^{\left(  1\right)  }}\cdot\ldots\cdot e^{is_{n_{m}%
}^{\left(  m\right)  }}e^{it_{n_{1}}^{\left(  1\right)  }}\cdot\ldots\cdot
e^{it_{n_{m}}^{\left(  m\right)  }}\right\vert ^{p}\right)  ^{\frac{1}{p}}\\
&  =1
\end{align*}
for all $(s_{n_{1}}^{\left(  1\right)  },...,s_{n_{m}}^{\left(  m\right)
})\in\left(  \Omega_{M}^{N}\right)  ^{m}$.

Thus,
\[
\left(  \left(  \frac{1}{M}\right)  ^{Nm}\sum_{(r_{n_{1}}^{\left(  1\right)
},...,r_{n_{m}}^{\left(  m\right)  })\in\left(  T_{M}^{N}\right)  ^{m}%
}\left\vert \sum_{n_{1},\ldots,n_{m}=1}^{N}a_{n_{1}\ldots n_{m}}w_{n_{1}%
}^{\left(  1\right)  }\cdot\ldots\cdot w_{n_{m}}^{\left(  m\right)  }r_{n_{1}%
}^{\left(  1\right)  }\cdot\ldots\cdot r_{n_{m}}^{\left(  m\right)
}\right\vert ^{p}\right)  ^{\frac{1}{p}}=1,
\]
for all $(w_{n_{1}}^{\left(  1\right)  },...,w_{n_{m}}^{\left(  m\right)
})\in\left(  T_{M}^{N}\right)  ^{m}.$ Note that, there is $\tau$ a bijection
of $\left\{  1,...,M^{Nm}\right\}  $ on $\left(  T_{M}^{N}\right)  ^{m}$, and
then we can write%
\begin{align}
&  \left(  \left(  \frac{1}{M}\right)  ^{Nm}\sum_{(r_{n_{1}}^{\left(
1\right)  },...,r_{n_{m}}^{\left(  m\right)  })\in\left(  T_{M}^{N}\right)
^{m}}\left\vert \sum_{n_{1},\ldots,n_{m}=1}^{N}a_{n_{1}\ldots n_{m}}w_{n_{1}%
}^{\left(  1\right)  }\cdot\ldots\cdot w_{n_{m}}^{\left(  m\right)  }r_{n_{1}%
}^{\left(  1\right)  }\cdot\ldots\cdot r_{n_{m}}^{\left(  m\right)
}\right\vert ^{p}\right)  ^{\frac{1}{p}}\label{esto}\\
&  =\left(  \left(  \frac{1}{M}\right)  ^{Nm}\sum_{i=1}^{M^{Nm}}\left\vert
\sum_{n_{1},\ldots,n_{m}=1}^{N}a_{n_{1}\ldots n_{m}}\tau_{n_{1}\ldots n_{m}%
}^{\left(  i\right)  }w_{n_{1}}^{\left(  1\right)  }\cdot\ldots\cdot w_{n_{m}%
}^{\left(  m\right)  }\right\vert ^{p}\right)  ^{\frac{1}{p}}\nonumber\\
&  =1\nonumber
\end{align}
for all $(w_{n_{1}}^{\left(  1\right)  },...,w_{n_{m}}^{\left(  m\right)
})\in\left(  T_{M}^{N}\right)  ^{m}$.

Consider the $\left(  m+1\right)  $-linear form $T:X_{p^{\ast}}^{M^{Nm}}\times
X_{\infty}^{N}\times\cdot\cdot\cdot\times X_{\infty}^{N}\rightarrow\mathbb{C}$
given by
\[
T(e_{i},e_{n_{1}},...,e_{n_{m}})=\frac{a_{n_{1}\ldots n_{m}}\tau_{n_{1}\ldots
n_{m}}^{\left(  i\right)  }}{M^{\frac{Nm}{p}}},\quad n_{j}\in\left\{
1,...,N\right\}  \text{ and }i\in\left\{  1,...,M^{Nm}\right\}  .
\]
Note that $\Vert T\Vert_{\left\{  2,...,m+1\right\}  ,M}\leq1$. In fact, using
H\"{o}lder's inequality and the equality (\ref{esto}) we get
\begin{align*}
&  \left\vert \sum_{i=1}^{M^{Nm}}\sum_{n_{1},\ldots,n_{m}=1}^{N}%
T(e_{i},e_{n_{1}},...,e_{n_{m}})w_{n_{1}}^{\left(  1\right)  }\cdot\ldots\cdot
w_{n_{m}}^{\left(  m\right)  }\cdot z_{i}\right\vert \\
&  =\left\vert \sum_{i=1}^{M^{Nm}}\sum_{n_{1},\ldots,n_{m}=1}^{N}%
\frac{a_{n_{1}\ldots n_{m}}\tau_{n_{1}\ldots n_{m}}^{\left(  i\right)  }%
}{M^{\frac{Nm}{p}}}w_{n_{1}}^{\left(  1\right)  }\cdot\ldots\cdot w_{n_{m}%
}^{\left(  m\right)  }\cdot z_{i}\right\vert \\
&  \leq\left(  \sum_{i=1}^{M^{Nm}}\left\vert \sum_{n_{1},\ldots,n_{m}=1}%
^{N}\frac{a_{n_{1}\ldots n_{m}}\tau_{n_{1}\ldots n_{m}}^{\left(  i\right)  }%
}{M^{\frac{Nm}{p}}}w_{n_{1}}^{\left(  1\right)  }\cdot\ldots\cdot w_{n_{m}%
}^{\left(  m\right)  }\right\vert ^{p}\right)  ^{\frac{1}{p}}\cdot\left(
\sum_{i=1}^{M^{Nm}}\left\vert z_{i}\right\vert ^{p^{\ast}}\right)  ^{\frac
{1}{p^{\ast}}}\\
&  \leq\left(  \left(  \frac{1}{M}\right)  ^{Nm}\sum_{i=1}^{M^{Nm}}\left\vert
\sum_{n_{1},\ldots,n_{m}=1}^{N}a_{n_{1}\ldots n_{m}}\tau_{n_{1}\ldots n_{m}%
}^{\left(  i\right)  }w_{n_{1}}^{\left(  1\right)  }\cdot\ldots\cdot w_{n_{m}%
}^{\left(  m\right)  }\right\vert ^{p}\right)  ^{\frac{1}{p}}\\
&  \overset{(\ref{esto})}{=}1
\end{align*}
for all $(w_{n_{1}}^{\left(  1\right)  },...,w_{n_{m}}^{\left(  m\right)
})\in\left(  T_{M}^{N}\right)  ^{m}$, and $\left(  z_{i}\right)  _{i}\in
B_{X_{p^{\ast}}}$. Therefore, $\Vert T\Vert_{\left\{  2,...,m+1\right\}
,M}\leq1$.

Moreover, using the last inequality in (\ref{orlicz}), Lemma \ref{664} and the
above norm estimate, we have
\begin{align*}
\left(  \sum\limits_{n_{1},\ldots,n_{m}=1}^{N}\left\vert a_{n_{1},\ldots
,n_{m}}\right\vert ^{2}\right)  ^{1/2}  &  =\left(  \sum\limits_{n_{1}%
,\ldots,n_{m}=1}^{N}\left(  \sum_{i=1}^{M^{Nm}}\left\vert \left(  \frac{1}%
{M}\right)  ^{\frac{mN}{p}}a_{n_{1}\ldots n_{m}}\tau_{n_{1}\ldots n_{m}%
}^{\left(  i\right)  }\right\vert ^{p}\right)  ^{2/p}\right)  ^{\frac{1}{2}}\\
&  =\left(  \sum\limits_{n_{1},\ldots,n_{m}=1}^{N}\left(  \sum_{i=1}^{M^{Nm}%
}\left\vert T(e_{i},e_{n_{1}},...,e_{n_{m}})\right\vert ^{p}\right)
^{2/p}\right)  ^{\frac{1}{2}}\\
&  \leq C_{(2,...,2,p)}^{\sigma_{m+1},\left(  p^{\ast},\infty,...,\infty
\right)  \mathbb{C}}\Vert T\Vert\\
&  \leq C_{(2,...,2,p)}^{\sigma_{m+1},\left(  p^{\ast},\infty,...,\infty
\right)  \mathbb{C}}\cdot r_{M}^{-m}\Vert T\Vert_{\left\{  2,...,m+1\right\}
,M}\\
&  \leq C_{(2,...,2,p)}^{\sigma_{m+1},\left(  p^{\ast},\infty,...,\infty
\right)  \mathbb{C}}\cdot r_{M}^{-m}.
\end{align*}
Thus, the inequality follows and $S_{m,M,p}\leq C_{(2,...,2,p)}^{\sigma
_{m+1},\left(  p^{\ast},\infty,...,\infty\right)  \mathbb{C}}\cdot r_{M}^{-m}$.
\end{proof}

Finally, making $M\rightarrow\infty$, Theorem \ref{665} recovers Theorem
\ref{d} and the estimate
\[
S_{m,M,p}\leq C_{(2,...,2,p)}^{\sigma_{m+1},\left(  p^{\ast},\infty
,...,\infty\right)  \mathbb{C}}\cdot r_{M}^{-m}%
\]
becomes
\[
\left(  S_{m,p}\right)  ^{-1}\leq C_{(2,...,2,p)}^{\sigma_{m+1},\left(
p^{\ast},\infty,...,\infty\right)  \mathbb{C}}.
\]
Thus, the previous inequality combined with Proposition \ref{multiracsam} and
(\ref{ppkk}) give us%
\[
\left(  S_{m,p}\right)  ^{-1}\leq C_{(2,...,2,p)}^{\sigma_{m+1},\left(
p^{\ast},\infty,...,\infty\right)  \mathbb{C}}\leq C_{(p,2,..,2)}^{id,\left(
p^{\ast},\infty,...,\infty\right)  \mathbb{C}}\leq\left(  S_{m,p}\right)
^{-1}%
\]
for all ~$m\geq1$ and for all $p\in\left[  1,2\right]  .$ Finally, by
Proposition \ref{e} we have
\[
C_{(2,...,2,p^{\ast})}^{\sigma_{m+1},\left(  p,\infty,...,\infty\right)
\mathbb{C}}=C_{(p^{\ast},2,..,2)}^{id,\left(  p,\infty,...,\infty\right)
\mathbb{C}}=\left(  \widetilde{A_{p^{\ast}}}\right)  ^{-\left(  m-1\right)  }%
\]
for all $m\geq2$ and for all $p\in\left[  2,\infty\right]  ,$ as we announced.
The special case $m=2$ and $p=\infty$ was previously obtained in \cite[page
31]{blei}.

\end{document}